\documentclass[letterpaper, 10 pt, conference]{ieeeconf}  

\IEEEoverridecommandlockouts                              
\usepackage{amssymb}
\usepackage{tikz-cd}
\usetikzlibrary{arrows}

\usepackage{amsthm}
\usepackage{amsmath}
\usepackage{graphicx}
\usepackage{subcaption} 
\usepackage{enumerate}

\usepackage{centernot}

\newtheoremstyle{mystyle}
  {}
  {}
  {\itshape}
  {}
  {\bfseries}
  {.}
  { }
  {}

\theoremstyle{mystyle}

\newtheorem{theorem}{Theorem}[section]

\newtheorem{definition}[theorem]{Definition}
\newtheorem{lemma}[theorem]{Lemma}

\newtheorem{corollary}[theorem]{Corollary}
\newtheorem{conjecture}[theorem]{Conjecture}

\newtheorem{example}[theorem]{Example}

\newtheorem{assumption}[theorem]{Assumption}

\newcommand\xqed[1]{%
  \leavevmode\unskip\penalty9999 \hbox{}\nobreak\hfill
  \quad\hbox{#1}}
\newcommand\exampleEnd{\xqed{$\circ$}}

\usepackage[backend=bibtex,
style=ieee,
bibencoding=ascii,
]{biblatex}
\makeatletter
\newbox\dottedarrow@box
\setbox\dottedarrow@box\hbox
  {%
    \begin{tikzpicture}
      \draw[dotted,->] (0,0) -- (1.5em,0);
    \end{tikzpicture}%
  }
\newcommand*\dottedarrow
  {\relax\ifmmode\expandafter\dottedarrow@m\else\expandafter\dottedarrow@t\fi}
\newcommand*\dottedarrow@t[1][1.5em]
  {\resizebox{#1}{!}{\raisebox{.5ex}{\usebox\dottedarrow@box}}}
\newcommand*\dottedarrow@m[1][]
  {%
    \if\relax\detokenize{#1}\relax
      \mathchoice
        {\dottedarrow@t}
        {\dottedarrow@t}
        {\dottedarrow@t[1.1em]}
        {\dottedarrow@t[0.9em]}%
    \else
      \dottedarrow@t[#1]%
    \fi
  }
\makeatother

\title{\LARGE \bf
Ordering and refining path-complete Lyapunov functions through\\ composition lifts
}

\author{Wouter Jongeneel and Rapha\"el M. Jungers
\thanks{RJ is a FNRS honorary Research Associate, supported by the Innoviris Foundation and the FNRS (Chist-Era Druid-net). This project received funding from the European Research Council (ERC)
under the European Union’s Horizon 2020 research and innovation programme, grant agreement No. 864017 (L2C).
The authors are with UCLouvain (ICTEAM). Contact: \{wouter.jongeneel,
raphael.jungers\}@uclouvain.be.} 
}

\begin{document}

\maketitle
\thispagestyle{empty}
\pagestyle{empty}

\begin{abstract}
A fruitful approach to study stability of switched systems is to look for multiple Lyapunov functions. However, in general, we do not yet understand the interplay between the desired stability certificate, the template of the Lyapunov functions and their mutual relationships to accommodate switching. 
In this work we elaborate on path-complete Lyapunov functions: a graphical framework that aims to elucidate this interplay.
In particular, previously, several preorders were introduced to compare multiple Lyapunov functions. These preorders are initially algorithmically intractable due to the algebraic nature of Lyapunov inequalities, yet, lifting techniques were proposed to turn some preorders purely combinatorial and thereby eventually tractable. In this note we show that a conjecture in this area regarding the so-called composition lift, that was believed to be true, is false. 
This refutal, however, points us to a beneficial structural feature of the composition lift that we exploit to iteratively refine path-complete graphs, plus, it points us to a favourable adaptation of the composition lift.
\end{abstract}

\section{Introduction}
Switched systems can be motivated based on physical grounds (\textit{e.g.}, think of switching gears, a walking robot or packet loss), or topological grounds (\textit{e.g.}, think of local or global obstructions to the existence of continuous feedback). For references we point the reader to~\cite{ref:liberzon2003switching,ref:tabuada2009verification,ref:goebel2012hybrid,ref:sanfelice2020hybrid,ref:jongeneel2023topological}.   
Concretely, we are interested in studying the stability of the origin $0\in \mathbb{R}^n$ under a discrete-time switched system of the form
\begin{equation}
\label{equ:switched}
    x(k+1) = f_{\sigma(k)}(x(k)),\quad k\in \mathbb{N}_{\geq 0},
\end{equation}
where $\sigma:\mathbb{N}_{\geq 0}\to \Sigma$ is the switching signal, for $\Sigma $ a finite alphabet that parametrizes a set of maps $F=\{f_{i}\in C^0(\mathbb{R}^n;\mathbb{R}^n)\,|\,f_i(0)=0,\,i\in \Sigma\}$, sometimes shortened to $\{f_i\,|\,i\in \Sigma\}$. We will assume that $\sigma$ is random and we study in that sense worst-case behaviour of the switched system~\eqref{equ:switched}. To be precise, we are interested in understanding if $0$ is \textit{uniformly globally asymptotically stable} (UGAS) under~\eqref{equ:switched}.

Stability of switched systems is theoretically interesting and non-trivial due to the fact that switching between stable maps $f_1,\dots,f_{|\Sigma|}$ can result in an unstable system and even more counter-intuitive, the combination of some unstable maps might be stable, although not under random switching.   

Although switched systems do admit a converse Lyapunov theory, \textit{e.g.}, see~\cite[Thm. 2.2]{ref:jungers2009joint}, looking for \textit{multiple} Lyapunov functions and an appropriate switching scheme between them is ought to be more fruitful than looking for a common Lyapunov function~\cite{ref:branicky1998multiple}. Nevertheless, the search for multiple Lyapunov functions is somewhat of an art and with this work we aim to contribute to making this a principled study. In particular, we look at \textit{path-complete Lyapunov functions} (PCLFs) \cite{ref:ahmadi2014joint}, a framework that aims to streamline the study of multiple Lyapunov functions by looking at the problem through the lens of graph- and order theory. Specifically, given a set of (multiple) Lyapunov functions $\{V_s\,|\,s\in S\}$, then each Lyapunov inequality corresponding to the switched system~\eqref{equ:switched} will be of the form $V_a(x)\geq \gamma V_b(f_i(x))$, for some $\gamma>0$, which can be understood as a labeled edge $(a,b,i)\in E$, being part of a directed graph (digraph) $\mathcal{G}=(S,E)$. We elaborate on this below. Throughout, we will assume that our graphs are always of the so-called ``\textit{expanded form}'' (\textit{e.g.}, all edges have labels of length one)~\cite[Def. 2.1]{ref:ahmadi2014joint}.    

\begin{definition}[Path-completeness]
    Consider a labeled digraph $\mathcal{G}=(S,E)$ on a finite alphabet $\Sigma$. Then, $\mathcal{G}$ is said to be path-complete, denoted $\mathcal{G}\in \mathrm{pc}(\Sigma)$, when for any finite word $i_1i_2\cdots i_k\in \Sigma^{k}$, there exists a sequence of nodes $s_1,\dots,s_{k+1} \in S$ such that $(s_j,s_{j+1},i_j)\in E$ $\forall j\in [k]$.  
\end{definition}
For now, Lyapunov functions used to assert stability of $0\in \mathbb{R}^n$ under~\eqref{equ:switched}, live in $C^0(\mathbb{R}^n;\mathbb{R}_{\geq 0})$, are positive definite (\textit{i.e.}, $V(0)=0$ and $V(x)>0$ $\forall\,x\in \mathbb{R}^n\setminus \{0\}$) and such that $\alpha_1(\|x\|_2)\leq V(x)\leq \alpha_2(\|x\|_2)$ $\forall\, x\in \mathbb{R}^n$ for $\mathcal{K}_{\infty}$ functions $\alpha_1$ and $\alpha_2$. We denote this set of \textit{candidate Lyapunov functions}, with respect to $0\in \mathbb{R}^n$, by $\mathrm{Lyap}_0(\mathbb{R}^n)$.

\begin{definition}[Path-complete Lyapunov functions]
\label{def:PCLF}
    Consider a switched system~\eqref{equ:switched} comprised by $F=\{f_i\,|\,i\in \Sigma\}$ and let $\mathcal{V}$ be a template of functions. Then, a path-complete Lyapunov function (PCLF) for $F$ is the triple $(\mathcal{G},V_S,\gamma)$ where $\mathcal{G}\in \mathrm{pc}(\Sigma)$, $V_S=\{V_s\in \mathrm{Lyap}_0(\mathbb{R}^n)\,|\,s\in S\}\subseteq \mathcal{V}$ and $\gamma>0$ such that for all edges $(a,b,i)$ of $\mathcal{G}$ we have that
    \begin{equation}
    \label{equ:pclf:ineq}
        V_a(x)\geq \gamma V_b(f_i(x)) \quad \forall x\in \mathbb{R}^n. 
    \end{equation}
    If this holds true, we say that the pair $(V_S,\gamma)$ is admissible for $\mathcal{G}$, $\mathcal{V}$ and $F$, denoted by $(V_S,\gamma)\in \mathrm{pclf}(\mathcal{G},\mathcal{V},F)$.
\end{definition}

We emphasize that the effective template under consideration is of the form $\mathcal{V}\cap \cup_{n\in \mathbb{N}_{>0}}\mathrm{Lyap}_0(\mathbb{R}^n)$, this, to clearly differentiate between structure due to the template and due to stability. Also, $\gamma$ is introduced to allow for a general stability analysis (\textit{i.e.}, readily generalizable beyond linear systems). 

Note that a graph $\mathcal{G}$ can be utilized as a PCLF for systems of different dimension (\textit{e.g.}, a common Lyapunov function is always a graph with a single node). As such, we will consider function spaces of the form $\mathcal{F}=\cup_{n\in \mathbb{N}_{>0}}\{\mathcal{F}_n \subseteq C^0(\mathbb{R}^n;\mathbb{R}^n)\}$ and similarly $\mathcal{V}=\cup_{n\in \mathbb{N}_{>0}} \mathcal{V}_n$. We do add that $F\subseteq \mathcal{F}$ is understood as a set contained in some $\mathcal{F}_n$.  

Now, path-completeness naturally connects to stability as follows (\textit{e.g.}, this is a natural extension of \cite[Thm. 2.4]{ref:ahmadi2014joint}).
\begin{theorem}[Uniform global asymptotic stability]
\label{thm:UGAS}
Let $(V_S,\gamma)\in \mathrm{pclf}(\mathcal{G},\mathcal{V},F)$ for a $\mathcal{G}\in \mathrm{pc}(\Sigma)$, template $\mathcal{V}$ and set $F$ parametrizing~\eqref{equ:switched}, then $0$ is UGAS under~\eqref{equ:switched} if $\gamma>1$. 
\end{theorem}  

Below we comment on a more informative version of Theorem~\ref{thm:UGAS} for linear systems, but first we highlight where we aim to contribute in the study of multiple Lyapunov functions. We are inspired by the open problem posed in~\cite[Sec. 7]{ref:ahmadi2014joint}: ``\textit{How can we compare the performance of different path-complete graphs in a systematic way?}'' More concretely, consider the following example. 
\begin{example}[Ambiguity in selecting inequalities]
\label{ex:ambiguity}
\begin{upshape}
Suppose we have a switched system with $\Sigma=\{1,2\}$. Besides trying to find a common Lyapunov function, one could consider 
\begin{equation}
\label{equ:Lyap:ineq:ex}
\begin{aligned}
V_{b}(x)\geq V_{b} (f_1(x))\\
V_{b}(x)\geq V_{c} (f_2(x))\\
V_{c}(x)\geq V_{c} (f_2(x))\\
V_{c}(x)\geq V_{b} (f_1(x))
\end{aligned}
\quad \text{or} \quad
\begin{aligned}
V_{b'}(x)\geq V_{b'} (f_1(x))\\
V_{b'}(x)\geq V_{c'} (f_1(x))\\
V_{c'}(x)\geq V_{c'} (f_2(x))\\
V_{c'}(x)\geq V_{b'} (f_2(x))
\end{aligned},
\end{equation}
corresponding to the path-complete graphs $\mathcal{G}_1$ and $\mathcal{G}_2$, in Figure~\ref{fig:PCLF1}. Similarly, one could consider the inequalities corresponding to Figure~\ref{fig:conjecture:counter}. 
However, towards a principled study of multiple Lyapunov functions, we lack appropriate machinery to determine which set of inequalities to prefer (\textit{e.g.}, what should one enter into a solver), if any? \exampleEnd
\end{upshape}
\end{example}  

To study the ambiguity as illustrated by Example~\ref{ex:ambiguity}, we use the following preorder, which is a slight modification of previous work to get a better grip on stability, \textit{e.g.}, see~\cite{ref:philippe2019completeV,ref:debauche2024path}, as such we use ``$\preceq$'' instead of ``$\leq$''.

\begin{definition}[Preorder relation]
\label{def:preorder}
Consider $\mathcal{G}_1, \mathcal{G}_2\in \mathrm{pc}(\Sigma)$, $\Gamma\subseteq \mathbb{R}_{>0}$, a template of Lyapunov functions $\mathcal{V}$ and a set of maps $\mathcal{F}$. If $\forall F\in \mathcal{F}^{|\Sigma|}$ $\exists (V_{S_1},\gamma\in \Gamma)\in \mathrm{pclf}(\mathcal{G}_1,\mathcal{V},F) \implies   \exists (V_{S_2},\gamma\in \Gamma)\in  \mathrm{pclf}(\mathcal{G}_2,\mathcal{V},F)$  
then
\begin{equation}
\label{equ:order}
\mathcal{G}_1\preceq_{(\Gamma,\mathcal{V},\mathcal{F})}\mathcal{G}_2.
\end{equation}
\end{definition}


One can interpret~\eqref{equ:order} as $\mathcal{G}_2$ being ``\textit{more expressive}'' than $\mathcal{G}_1$ and in that sense $\mathcal{G}_2$ is understood as ``\textit{better}''. Clearly, the minimal element would be a common Lyapunov function $V_a\in \mathcal{V}$ together with the graph $\mathcal{G}_0=(S_0,E_0)$ comprised out of the single node $a=:S_0$ and the edges $E_0:=\{(a,a,i)\,|\,i\in \Sigma \}$, \textit{e.g.}, see Figure~\ref{fig:connected}. However, as we know, this hinges on our choice of template (\textit{e.g.}, there are stable switched systems that fail to admit a common quadratic Lyapunov function). If the preorder\footnote{Note that the relation~\eqref{equ:order} clearly satisfies the \textit{reflexive} and \textit{transitive} properties of an order relation. Yet, \textit{antisymmetry} does not hold true in general \textit{cf}. Figure~\ref{fig:PCLF1}, thus we can only speak of a \textit{preorder}.} relation~\eqref{equ:order} must be true for any $\mathcal{F}$ we write $\mathcal{G}_1\preceq_{\Gamma,\mathcal{V}} \mathcal{G}_2$, on the other hand, if~\eqref{equ:order} must hold for any $\mathcal{V}$, $\mathcal{F}$ and $\Gamma$, we write $\mathcal{G}_1\preceq \mathcal{G}_2$.   

Direct use of Definition~\ref{def:preorder} is as follows. Suppose that $F_{\mathcal{A}}$ is the set of linear maps parametrized by the matrices $\mathcal{A}:=\{A_1,\dots,A_{|\Sigma|}\}$. Stability of a switched linear system comprised by $\mathcal{A}$, under arbitrary switching, is captured by the \textit{joint spectral radius} (JSR)~\cite{ref:rota1960} of $\mathcal{A}$, defined through
\begin{align*}
\rho(\mathcal{A}):= \lim_{k\to +\infty}\sup_{ \{j_1,\dots,j_k\}\in \Sigma^k  } \|A_{j_k}A_{j_{k-1}}\cdots A_{j_1}\|^{1/k}, 
\end{align*}
not being larger than $1$. Unfortunately, computing $\rho(\mathcal{A})$ is NP-hard and checking whether $\rho(\mathcal{A})\leq 1$ is undecidable and not-algebraic, \textit{e.g.}, see~\cite{ref:kozyakin1990algebraic,ref:tsitsiklis1997lyapunov,ref:blondel2000boundedness, ref:jungers2009joint}. Hence, one is drawn to the development of approximation schemes, \textit{e.g.}, see~\cite{ref:parrilo2008approximation}.  

Next, we provide the linear version of Theorem~\ref{thm:UGAS}. 

\begin{theorem}[The joint spectral radius~{\cite[Prop. 1.2]{ref:jungers2009joint}}]
\label{thm:pclf:JSR}
For $\mathcal{G}\in \mathrm{pc}(\Sigma)$, with $F_{\mathcal{A}}$ parametrized by $\mathcal{A}:=\{A_1,\dots,A_{|\Sigma|}\}$, let $(V_S,1)\in \mathrm{pclf}(\mathcal{G},\mathcal{V},F_{r^{-1} \mathcal{A}})$ for some template $\mathcal{V}$ and some $r>0$, then $\rho(\mathcal{A})\leq r$. 
\end{theorem}    

With Theorem~\ref{thm:pclf:JSR} in mind, we define the graph- and template-based approximation of $\rho(\mathcal{A})$ as
\begin{align}
\label{equ:gamma:opt}
\rho_{\mathcal{G},\mathcal{V}}(\mathcal{A}):= \min_{r>0}\{r:\exists (V,1)\in \mathrm{pclf}(\mathcal{G},\mathcal{V},F_{r^{-1}\mathcal{A}})\}.
\end{align}
Then, $\rho(\mathcal{A})\leq \rho_{\mathcal{G},\mathcal{V}}(\mathcal{A})\leq r'$ for any $r'>0$ feasible in~\eqref{equ:gamma:opt} and in particular, $\mathcal{G}_1\preceq_{(\{1\},\mathcal{V},\mathcal{F})}\mathcal{G}_2 \implies \rho_{\mathcal{G}_2,\mathcal{V}}(\mathcal{A})\leq \rho_{\mathcal{G}_1,\mathcal{V}}(\mathcal{A})$ \cite[Prop. 6.22]{ref:debauche2024path}, \textit{i.e.}, the preorder identifies which graph is better, when approximating the JSR. The preorder ``$\preceq_{(\{1\},\cdot)}$'' will be denoted by ``$\leq_{(\cdot)}$'', in line with \cite{ref:debauche2024path}. 

At last, we recall that a graph $\mathcal{G}_1=(S_1,E_1)$ is said to \textit{simulate} a graph $\mathcal{G}_2=(S_2,E_2)$ when there is a map $R:S_2\to S_1$ such that $(a,b,i)\in E_2 \implies (R(a),R(b),i)\in E_1$ (\textit{e.g.}, think of the behaviour of $\mathcal{G}_2$ being ``\textit{contained}'' in $\mathcal{G}_1$). It is known that $\mathcal{G}_1\leq \mathcal{G}_2\iff$ $\mathcal{G}_1$ simulates $\mathcal{G}_2$~\cite{ref:philippe2019completeV}. As $\mathcal{G}_1\leq \mathcal{G}_2$ is a strong demand, researchers looked at different preorders and ``\textit{lifted}'' simulation relations, as discussed next.

\section{The sum lift}
\label{sec:prelim}
Consider the Lyapunov inequalities $V_a(x)\geq V_b(f_i(x))$ and $V_c(x)\geq V_d(f_i(x))$, then define $V_{\{a,c\}}:=V_a+V_c$ and $V_{\{b,d\}}:=V_b+V_d$, we readily obtain the new inequality $V_{\{a,c\}}(x)\geq V_{\{b,d\}}(f_i(x))$, so, it appears natural to consider $\mathcal{V}$ to be closed under addition. This is the \textit{sum lift} \cite[Sec. 7.3]{ref:debauche2024path}. Let $\mathrm{mset}^T(S)$ be the length-$T$ multiset of elements of $S$ (\textit{e.g.}, $\mathrm{mset}^2(\{a,b\})=\{ \{a,a\},\{a,b\},\{b,b\}\}$). 

\begin{definition}[The $T$-sum lift]
\label{def:T:sum:lift}
Consider a labeled digraph $\mathcal{G}=(S,E)$ on $\Sigma$, then for any $T\in \mathbb{N}_{>0}$ we define the {$T$-sum lift} of $\mathcal{G}$, denoted $\mathcal{G}^{\oplus T}=(S^{\oplus T},E^{\oplus T})$, through  $S^{\oplus T}:=\mathrm{mset}^T(S)$; and  $E^{\oplus T}:=\{((a_1,\dots,a_T),(b_1,\dots,b_T),i)\,|\,(a_k,b_k,i)\in E,\,\forall k\in [T]\}$. 
\end{definition}
Then, the \textit{sum lift} of $\mathcal{G}$ becomes $\mathcal{G}^{\oplus}=\cup_{T\in \mathbb{N}_{>0}}\mathcal{G}^{\oplus T}$.The sum lift is particularly interesting due to the following result. 

\begin{theorem}[$T$-sum lift simulation relation {\cite[Thm. 1]{ref:debauche2023characterizationV}}]
\label{theorem:T:sum:lift}
    Let $\mathcal{G}_1$ and $\mathcal{G}_2$ be path-complete on $\Sigma$, then the following statements are equivalent:
    \begin{enumerate}[(i)]
        \item $\mathcal{G}_1^{\oplus}$ simulates $\mathcal{G}_2$; 
        \item $\mathcal{G}_1\leq_{(\mathcal{V})}\mathcal{G}_2$ for any template $\mathcal{V}$ closed under addition. 
    \end{enumerate}
\end{theorem}
One might hope that Theorem~\ref{theorem:T:sum:lift} holds when restricted to \textit{quadratic forms} and linear maps. We will show that there are graphs that can be ordered with respect to this template, yet, the ordering does not imply the existence of a $T$-sum lift simulation relation. This shows that Theorem~\ref{theorem:T:sum:lift} fails subject to these restrictions. This motivates new lifts.

Here, we can build upon \cite[Ex. 3.9]{ref:philippe2019completeV}. It can be shown that there is no simulation relation between $\mathcal{G}_1=(S_1,E_1)$ and $\mathcal{G}_2=(S_2,E_2)$ as drawn in Figure~\ref{fig:PCLF1}, that is, in both directions. The key observation is as follows. Suppose that $\mathcal{G}_1$ simulates $\mathcal{G}_2$, then there should be a map $R:S_2\to S_1$ such that for all $(a,b,i)\in E_2$ we have that $(R(a),R(b),i)\in E_1$, however, this fails as we must have $b'\mapsto R(b')=b$ and $c'\mapsto R(c')=c$, yet, $(b',c',1)\mapsto (b,c,1)\notin E_1$. 

Now, suppose that the underlying switched system is comprised out of \textit{invertible} linear maps, for the moment, we denote all those maps by $\mathcal{F}$. Then, it follows that $\mathcal{G}_1$ and $\mathcal{G}_2$ can be ordered when the template $\mathcal{V}$ is closed under \textit{composition} with elements from $\mathcal{F}$ (\textit{e.g.}, like quadratic forms). We point out that $\mathcal{G}_1$ and $\mathcal{G}_2$ are in this sense \textit{equivalent} since $\mathcal{G}_1\preceq_{(\mathcal{V},\mathcal{F})}\mathcal{G}_2$ and $\mathcal{G}_2\preceq_{(\mathcal{V},\mathcal{F})}\mathcal{G}_1$ (\textit{n.b.}, the ordering holds for any $\Gamma\subseteq \mathbb{R}^n_{>0}$). To see this, suppose that $(\{V_b,V_c\},\gamma)$ is valid for $\mathcal{G}_1$, then we can pick $V_{b'}:=V_b\circ A_1$, $V_{c'}:=V_c\circ A_2$ and $\gamma$ for $\mathcal{G}_2$. We will see that this relates to the \textit{forward} composition lift. On the other hand, suppose that $(\{V_{b'},V_{c'}\},\gamma)$ is valid for $\mathcal{G}_2$, then we can pick $V_{b}:=V_{b'}\circ A_1^{-1}$, $V_{c}:=V_{c'}\circ A_2^{-1}$ and $\gamma$ for $\mathcal{G}_1$. We will see that this relates to the \textit{backward} composition lift.

\begin{figure}
    \centering
    \includegraphics[scale=0.6]{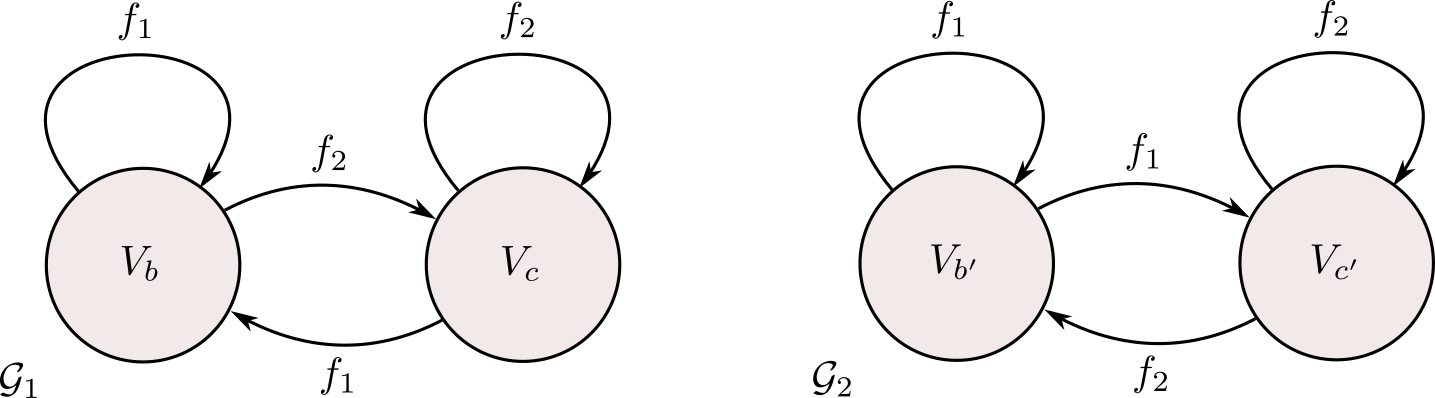}
    \caption{The graphs $\mathcal{G}_1$ and $\mathcal{G}_2$ as in \cite[Ex. 3.9]{ref:philippe2019completeV}.}
    \label{fig:PCLF1}
\end{figure}

We claim, however, that there is no way to construct a $T$-sum lift of either $\mathcal{G}_1$ or $\mathcal{G}_2$ to simulate the other. 

\begin{example}[No $T$-sum lift simulation relation for the graphs in Figure~\ref{fig:PCLF1}]
\label{ex:counter:T-sum}
\begin{upshape}
 Suppose that $\mathcal{G}_1^{\oplus T}=(S_1^{\oplus T},E_1^{\oplus T})$ simulates $\mathcal{G}_2=(S_2,E_2)$, then, there must be a map $R:S_2\to S_1^{\oplus T}$ such that for any $(a,b,i)\in E_2$ we have $(R(a),R(b),i)\in E_1^{\oplus T}$. See that $E_2$ is comprised out of $(b',b',1)$, $(c',c',2)$, $(b',c',1)$ and $(c',b',2)$. Note, we have two-self loops, \textit{e.g.}, we will have $(b',b',1)\mapsto (R(b'),R(b'),1)$. Now consider elements of $E_1^{\oplus T}$, they are of the form $(\{a_1,\dots,a_T\},\{b_1,\dots,b_T\},i)$ with $(a_k,b_k,i)\in E_1$ for all $k\in [T]$ and $\{a_1,\dots,a_T\},\{b_1,\dots,b_T\}\in \mathrm{mset}^T(S_1)$. Given the structure of $E_2$ and $E_1$, we must have $b'\mapsto R(b'):=\{b,\dots,b\}\in \mathrm{mset}^T(S_1)$ and $c'\mapsto R(c'):=\{c,\dots,c\}\in \mathrm{mset}^T(S_1)$. However, then we must also have $(b',c',1)\mapsto (R(b'),R(c'),1)=(\{b,\dots,b\},\{c,\dots,c\},1)\in E_1^{\oplus T}$, yet, $(b,c,1)\notin E_1$, which is a contradiction. This extends to $\mathcal{G}_1^{\oplus}$ not simulating $\mathcal{G}_2$, \textit{e.g.}, by~\cite[Lem. 7.20]{ref:debauche2024path}.
 \exampleEnd    
 \end{upshape}
\end{example}


\begin{corollary}[Quadratic forms cannot be ordered through a sum lift]
\label{cor:T:sum:lift:quadratic}
Theorem~\ref{theorem:T:sum:lift} fails when we restrict to quadratic forms and invertible linear maps, \textit{i.e.}, then $(ii)\centernot\implies(i)$. 
\end{corollary}
Corollary~\ref{cor:T:sum:lift:quadratic} raises the question: ``\textit{Is there a lift suitable for the quadratic template and the graphs from Figure~\ref{fig:PCLF1}?}'' 

\section{The composition lift}
Suppose we are given a Lyapunov inequality $V_a(x)\geq V_b(f_i(x))$, \textit{e.g.}, an edge $(a,b,i)\in E$. Now, given some $f_j\in F$, we readily construct the following inequality $V_a(f_j(x))\geq V_b (f_i(f_j(x))$ by substituting $f_j(x)$ for $x$. Then, if we define $V_{(a,j)}:= V_a\circ f_j$ and similarly $V_{(b,i)}:=V_b\circ f_i$ we obtain the new inequality $V_{(a,j)}(x)\geq V_{(b,i)}(f_j(x))$. Now, one might say that it is natural to consider templates that are closed under composition with the dynamics. Utilizing these newly formed inequalities is the \textit{composition lift} \cite[Sec. 7.5]{ref:debauche2024path}. 

\begin{definition}[Composition lifts]
\label{def:composition}
Consider a labeled digraph $\mathcal{G}=(S,E)$ on $\Sigma$. For any $T\in \mathbb{N}_{>0}$, the {$T$-forward composition lift} of $\mathcal{G}$, denoted $\mathcal{G}^{\circ T}=(S^{\circ T},E^{\circ T})$, is defined by (i) $S^{\circ T}:=\{(s,j_1,\dots,j_T)\,|\,s\in S,\,(j_1,\dots,j_T)\in \Sigma^T\}$; and (ii) all $e\in E^{\circ T}$ are such that $e=( (a,j_1,\dots,j_T), (b,i,j_1,\dots,j_{T-1}),j_T)$ for some $(a,b,i)\in E$ and $(j_1,\dots,j_T)\in \Sigma^T$. Let $\mathcal{G}^{\circ 0}:=\mathcal{G}$, then, $\mathcal{G}^{\circ}={\cup}_{T\in \mathbb{N}_{\geq 0}}\mathcal{G}^{\circ T}$ is the forward composition lift. Similarly, the {$T$-backward composition lift} of $\mathcal{G}$, denoted $\mathcal{G}^{\circ - T}=(S^{\circ - T},E^{\circ -T})$, is defined through $S^{\circ -T}:=S^{\circ T}$ with elements of $E^{\circ - T}$ being of the form $((a,i,j_1,\dots,j_{T-1}),(b,j_1,\dots,j_T),j_T)$.  
\end{definition}

\subsection{Properties of the composition lift}
\label{sec:comp:lift:properties}
Note that Definition~\ref{def:composition} is not directly related to stability, as clarified below (\textit{e.g.}, $V_s\circ f_j\notin \mathrm{Lyap}_0(\mathbb{R}^n)$ if $f_j=0$). 

\begin{figure}
    \centering
    \includegraphics[scale=0.6]{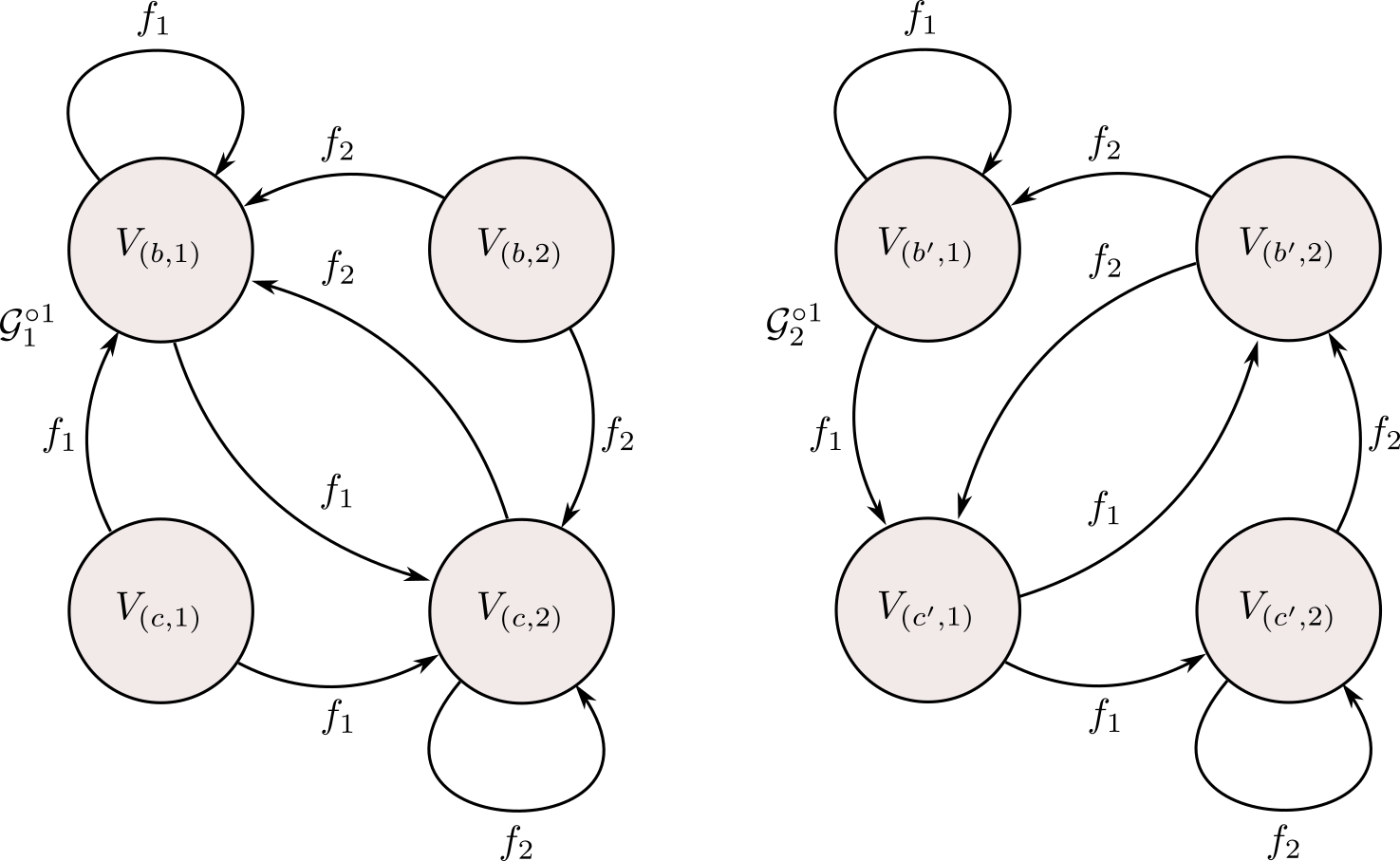}
    \caption{The graphs $\mathcal{G}^{\circ 1}_1$, $\mathcal{G}_2^{\circ 1}$ corresponding to Example~\ref{ex:12lifts} (\textit{i.e.}, $1$-composition lifts of Figure~\ref{fig:PCLF1}).}
    \label{fig:PCLF12lifts}
\end{figure}
\begin{example}[Composition lifts]
\label{ex:12lifts}
\begin{upshape}
First, we construct the forward- and backward composition lift of $\mathcal{G}_0$ from Figure~\ref{fig:connected}, see that $\mathcal{G}_0^{\circ 1}=\mathcal{G}_2$ and $\mathcal{G}_0^{\circ -1}=\mathcal{G}_1$, for $\mathcal{G}_1$ and $\mathcal{G}_2$ as in Figure~\ref{fig:PCLF1}. Note that a $T$-sum lift of $\mathcal{G}_0$ would result in copies of $\mathcal{G}_0$. 
Secondly, we construct $\mathcal{G}_1^{\circ 1}$ and $\mathcal{G}_2^{\circ 1}$, as shown in Figure~\ref{fig:PCLF12lifts}. See that $\mathcal{G}_1^{\circ 1}$ simulates $\mathcal{G}_2$, similarly, $\mathcal{G}_2^{\circ -1}$ simulates $\mathcal{G}_1$. This motivates our studies. \exampleEnd
\end{upshape}
\end{example}

By path-completeness, we can without loss of generality assume that $\mathcal{G}$ is connected. Interestingly, the $T$-sum lift does not preserve this topological property~\cite[Fig.~7.2]{ref:debauche2024path}. 
In contrast, $T$-forward composition lifts are always connected whenever $\mathcal{G}$ is connected and without sinks (\textit{i.e}, all nodes have outgoing edges), \textit{n.b.}, we merely discuss connectedness here, not graphs being \textit{strongly} connected, see Figure~\ref{fig:connected}.

\begin{figure}
    \centering
    \includegraphics[scale=0.6]{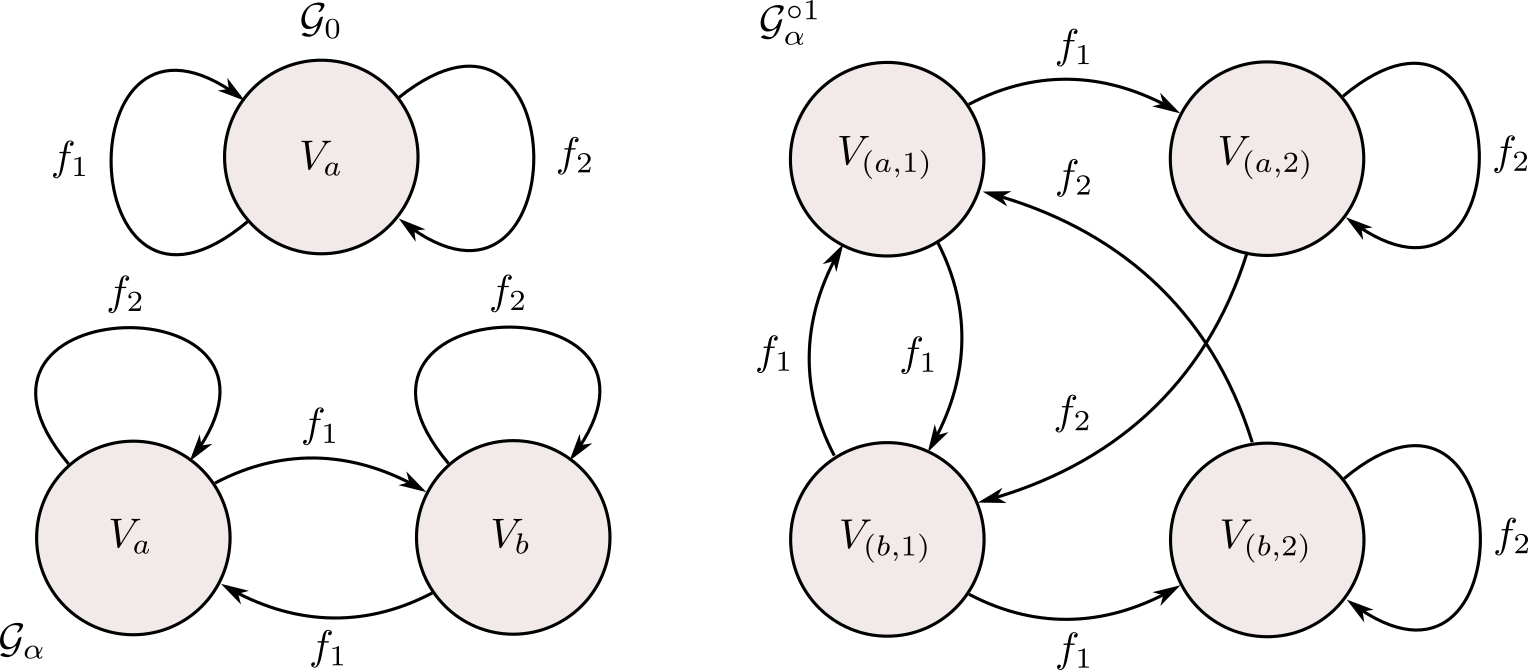}
    \caption{Even for strongly-connected graphs, the $T$-sum lift can result in a graph that fails to be connected (\textit{e.g.} $\mathcal{G}_{\alpha}^{\oplus 2}=\mathcal{G}_{\alpha}\sqcup \mathcal{G}_0$), this, in contrast to the $T$-composition lift.}
    \label{fig:connected}
\end{figure}

\begin{lemma}[$T$-composition lifts preserve connectedness]
\label{lem:connected}
Let $\mathcal{G}=(S,E)\in \mathrm{pc}(\Sigma)$ be connected and sink-free. Then, for any $T\in \mathbb{N}_{\geq 0}$, the graph $\mathcal{G}^{\circ T}$ is connected, similarly, if $\mathcal{G}$ is connected and source-free, then $\mathcal{G}^{\circ -T}$ is connected.
\end{lemma}
\begin{proof}
    The proof is by induction, we focus on the forward lift, the backward lift is then immediate by duality, \textit{i.e.}, $\mathcal{G}^{\circ -T}=((\mathcal{G}^{\mathsf{T}})^{\circ T})^{\mathsf{T}}$~\cite[Prop. 7.64]{ref:debauche2024path}. By assumption $\mathcal{G}=:\mathcal{G}^{\circ 0}$ is connected. Now, since $\mathcal{G}$ is sink-free, each node $s\in S$ has an outgoing edge, say $(s,b,i)\in E$. Thus, the set of nodes $\{(s,j)\}_{j\in \Sigma}\subset S^{\circ 1}$ is connected via the node $(b,i)$. Similarly, for the set of nodes $\{(b,j)\}_{j\in \Sigma}$. Since $\mathcal{G}$ itself is connected, all those sets of nodes are connected and thus $\mathcal{G}^{\circ 1}$ is connected. Note that $\mathcal{G}^{\circ 1}$ is again sink-free as for any $(a,j)\in S^{\circ 1}$ there must be a $(b,i)\in S^{\circ 1}$ such that $((a,j),(b,i),j)\in E^{\circ 1}$, this by $\mathcal{G}^0$ being sink-free. Then, see that $\mathcal{G}^{\circ(k+1)}=(\mathcal{G}^{\circ k})^{\circ 1}$ for any $k\in \mathbb{N}_{\geq 0}$, where we can interpret, by induction, $\mathcal{G}^{\circ k}$ as another connected graph.
\end{proof}

We remark that the sink-free assumption is there to avoid pathologies, akin to~\cite[Ass.~7.1]{ref:debauche2024path}.

\begin{assumption}[Non-redundant graphs]
\label{ass:pclf}
Path-complete graphs under consideration are strongly-connected and such that if we remove any edge, the graph is not
path-complete.
\end{assumption}


Ideally, the composition lift is ``\textit{valid}'', \textit{i.e.}, $\mathcal{G}\preceq_{\mathcal{V},\mathcal{F}}\mathcal{G}^{\circ}$ for any $\mathcal{V}$ that is closed under the forward composition with $\mathcal{F}$. When $\mathcal{F}$ contains non-invertible maps, this cannot be guaranteed, Lyapunov functions can lose coercivity after composition \textit{e.g.}, $V\circ f\notin \mathrm{Lyap}_0(\mathbb{R}^n)$ while $V\in \mathrm{Lyap}_0(\mathbb{R}^n)$, thus, we work with invertible maps. To further comment on the relevance of the composition lift, we elaborate on \cite{ref:bliman2003stability}.

\begin{lemma}[A converse Lyapunov lemma]
\label{lem:converse}
Let $0\in \mathbb{R}^n$ be UGAS under a switched system~\eqref{equ:switched}, with $F:=\{f_i\,|\,i\in \Sigma\}$ being some invertible linear maps, then, there is a set of multiple Lyapunov functions, contained in the forward compositional closure of $x\mapsto \|x\|_2^2$ under $F$, to assert this. 
\end{lemma}
\begin{proof}
Suppose that $F$ comprises a switched linear system rendering $0$ UGES (\textit{e.g.}, there are $M,\alpha>0$ such that $\|\varphi^k(x_0)\|^2_2\leq M e^{-\alpha k}\|x_0\|^2_2$, with $\varphi^k(x_0)$ denoting an arbitrary solution under~\eqref{equ:switched}, starting from $x_0$, after $k$ timesteps). We claim there are multiple Lyapunov functions within the forward compositional closure of the quadratic function $x\mapsto V_q(x):=\|x\|_2^2$ under $F$, that is, within $V_q\cup\{V_q\circ f_{i_1}\circ \cdots \circ f_{i_T}\,|\, f_{i_1},\dots,f_{i_T}\in F,\,T\in \mathbb{N}_{>0}\}$. To show this, we can directly follow~\cite{ref:bliman2003stability}. Specifically, by the UGES property we know there is a $K'\in \mathbb{N}_{\geq 0}$ such that $M e^{-\alpha K'}<1$ and thus as $\varphi^K(x_0)=A_{i_K}\cdots A_{i_1}x_0$ we get that $(A_{i_K}\cdots A_{i_1})^{\mathsf{T}}(A_{i_K}\cdots A_{i_1}) \prec I_n$ must hold for any $(i_K,\dots,i_1)\in \Sigma^K$ and $K\geq K'$. Thus, $V_q$ is a common Lyapunov function for all words of length $K\geq K'$. Now, fix $K$ and construct the ``\textit{expanded form}'' of this graph~\cite[Sec. 2]{ref:ahmadi2014joint}, and see that all nodes are elements of the forward compositional closure of $V_q$ (with no more than $K-1$ compositions). 
The claim follows by the equivalence of UGES and UGAS for linear switched systems~\cite{ref:bhaya1994equivalence}. 
\end{proof}

We already point out that the converse result from Lemma~\ref{lem:converse} indicates that the multiple Lyapunov functions need not naturally live in the \textit{same} $T$-forward composition lift, \textit{i.e.}, we might need elements from $\mathcal{G}^{\circ T_1}$ and $\mathcal{G}^{\circ T_2}$ for $T_1\neq T_2$. This observation is the crux of the next section.  

\section{Main results}
First, we provide a generalization of \cite{ref:philippe2019completeV}, implying that ``$\leq$'' and ``$\preceq$'' are inducing equivalent preorders on $\mathrm{pc}(\Sigma)$.
\begin{lemma}[Simulation relations and Definition~\ref{def:preorder}]
\label{lemma:simulation:order}
Let $\mathcal{G}_1,\mathcal{G}_2\in \mathrm{pc}(\Sigma)$, then, $\mathcal{G}_1$ simulates $\mathcal{G}_2$ if and only if $\mathcal{G}_1\preceq \mathcal{G}_2$. 
\end{lemma}
\begin{proof}
It is known for $\Gamma=\{1\}$~\cite[Thm. 3.5]{ref:philippe2019completeV}, adjusting the proof of \cite[Thm. 3.6]{ref:philippe2019completeV}, we can readily handle the factor $\gamma$, \textit{i.e.}, going from $x^{\mathsf{T}}P_a x \geq (A_ix)^{\mathsf{T}}P_b (A_i x)$ to $x^{\mathsf{T}}P_a x \geq \gamma (\gamma^{-1/2}A_ix)^{\mathsf{T}}P_b (\gamma^{-1/2}A_i x)$, for $P_a,P_b\succ 0$. Indeed, this shows slightly more than the statement.    
\end{proof}

\subsection{Simulation through the composition lift}
Initial work on the composition lift appeared in \cite[Prop. 4.2]{ref:ahmadi2014joint}, \cite[Ex. IV.11]{ref:philippe2018path} and \cite[Ex. 3.9]{ref:philippe2019completeV}, with the first formalizations appearing in \cite[Sec. 3.2]{ref:debauche2022comparison}. The most detailled study can be found in~\cite[Ch. 6-8]{ref:debauche2024path}. Ever since the formalization of lifts, the following conjecture was expected to be true. In particular, this is Conjecture 8.20 from~\cite{ref:debauche2024path}. 

\begin{conjecture}[The composition lift {\cite[Conj. 8.20]{ref:debauche2024path}}]
\label{conj:forward:composition}
Consider $\mathcal{G}, \widetilde{\mathcal{G}}\in \mathrm{pc}(\Sigma)$, satisfying Assumption~\ref{ass:pclf}. Suppose that $\mathcal{F}$ comprises a subset of linear invertible maps, then, the following two statements are equivalent 
\begin{enumerate}[(i)]
    \item $\mathcal{G}^{\circ}$ simulates $\widetilde{\mathcal{G}}$
    \item $\mathcal{G}\leq_{\mathcal{V},\mathcal{F}} \widetilde{\mathcal{G}}$ for any template $\mathcal{V}$ closed under forward composition with the class of dynamics $\mathcal{F}$.
\end{enumerate}
\end{conjecture}
To clarify, when we write ``\textit{$\mathcal{G}\leq_{\mathcal{V},\mathcal{F}}\widetilde{\mathcal{G}}$ for any template $\mathcal{V}$, closed under forward composition with $\mathcal{F}$}'', we follow~\cite{ref:debauche2024path} and mean the following. For any template $\mathcal{V}$ such that, $\forall n\in \mathbb{N}_{> 0}$, $V\in \mathcal{V}_n\implies V\circ f\in \mathcal{V}_n\,\forall f\in \mathcal{F}_n$, we have that for any $F\in \mathcal{F}^{|\Sigma|}$ $    \exists (V_{S}\subseteq \mathcal{V})\in \mathrm{pclf}(\mathcal{G},\mathcal{V},F) \implies   \exists (V_{\widetilde{S}}\subseteq \mathcal{V})\in  \mathrm{pclf}(\widetilde{\mathcal{G}},\mathcal{V},F)$. Now, regarding Conjecture~\ref{conj:forward:composition}, it is known that $(i)\implies (ii)$ is true for ``$\leq$'' \cite[Thm. 7.65]{ref:debauche2024path} (\textit{i.e.}, conditioned on using invertible maps) and additionally by Lemma~\ref{lemma:simulation:order} for ``$\preceq$'', recall also Figure~\ref{fig:PCLF1}. However, by means of a counterexample we show that $(ii)\implies (i)$ fails. We focus on the preorder ``$\leq $'' and linear maps for simplicity. 

\begin{figure}
    \centering
   \includegraphics[scale=0.6]{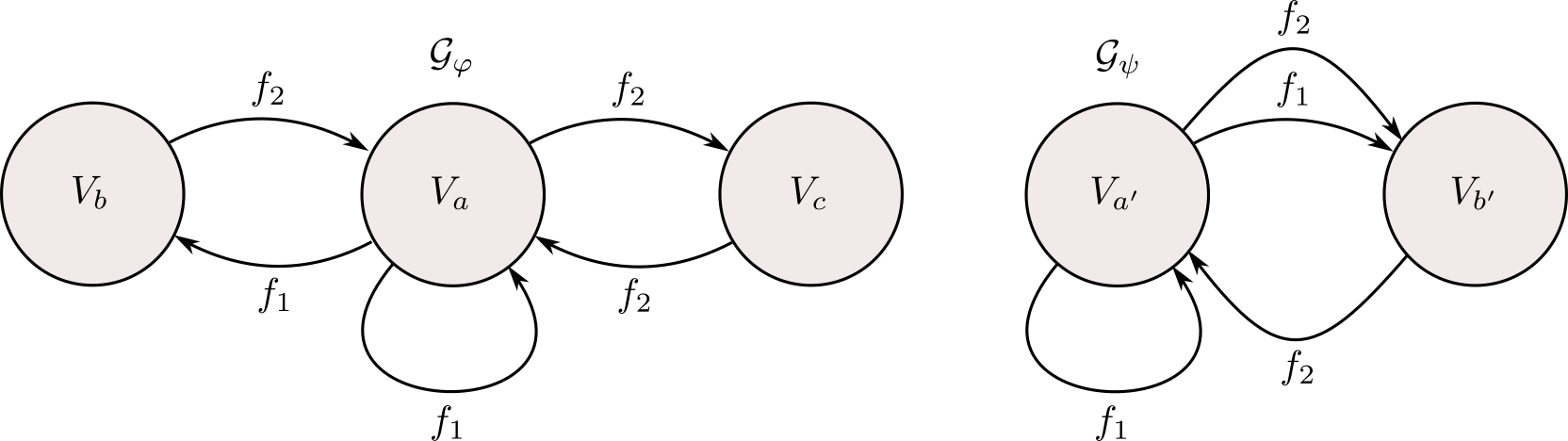}
    \caption{The graphs corresponding to Example~\ref{ex:conjecture:counter}.}
    \label{fig:conjecture:counter}
\end{figure}

\begin{example}[Conjecture~\ref{conj:forward:composition}: $(ii)\centernot\implies(i)$]
\label{ex:conjecture:counter}
\begin{upshape}
Consider the graphs $\mathcal{G}_{\varphi}$ and $\mathcal{G}_{\psi}$ from Figure~\ref{fig:conjecture:counter}. Both graphs are path-complete on $\Sigma=\{1,2\}$ and satisfy Assumption~\ref{ass:pclf}. In Step 1 we show that $\mathcal{G}_{\varphi}\leq_{\mathcal{V},\mathcal{F}}\mathcal{G}_{\psi}$ for any $\mathcal{V}$ that is closed under forward composition with $\mathcal{F}$. In Step 2 we show that despite Step 1, $\mathcal{G}_{\varphi}^{\circ}$ does \textit{not} simulate $\mathcal{G}_{\psi}$ and thus $(ii)\centernot\implies(i)$.

\textit{Step 1}. We claim that $\mathcal{G}_{\varphi}\leq_{\mathcal{V},\mathcal{F}}\mathcal{G}_{\psi}$ for any $\mathcal{V}$ that is closed under forward composition with the class of dynamics $\mathcal{F}$. To that end, let $F=\{f_i\,|\,i\in \Sigma\}$ be any two invertible linear maps parametrized by some matrices $A_1,A_2\in \mathbb{R}^{n\times n}$, for some $n\in \mathbb{N}_{> 0}$. Then, given any triple $\{V_a,V_b,V_c\}\in (\mathcal{V}_n\cap \mathrm{Lyap}_0(\mathbb{R}^n))^3$ of admissible Lyapunov functions for $\mathcal{G}_{\varphi}$, we claim that $x\mapsto {V}_{a'}(x):=V_a(x)$ and $x\mapsto {V}_{b'}(x):=V_a(A_2 x)$ are admissible for $\mathcal{G}_{\psi}$. Note that by assumption (\textit{i.e.}, closure of $\mathcal{V}$ under forward composition with the dynamics and $A_2\in \mathsf{GL}(n,\mathbb{R})$), $x\mapsto V_a(A_2 x)\in \mathcal{V}_n\cap \mathrm{Lyap}_0(\mathbb{R}^n)$. The inequalities to check for $\mathcal{G}_{\psi}$ are
\begin{align}
\label{ineq1} {V}_{a'}(x)&\geq {V}_{a'}(A_1 x),\quad \forall x\in \mathbb{R}^n,\\
\label{ineq2} {V}_{a'}(x)&\geq {V}_{b'}(A_1 x),\quad \forall x\in \mathbb{R}^n,\\
\label{ineq3} {V}_{a'}(x)&\geq {V}_{b'}(A_2 x),\quad \forall x\in \mathbb{R}^n,\\
\label{ineq4} {V}_{b'}(x)&\geq {V}_{a'}(A_2 x),\quad \forall x\in \mathbb{R}^n.
\end{align}
Inequality~\eqref{ineq1} is true since ${V}_{a'} = V_a$ and $V_a(x)\geq V_a(A_1 x)$ by $\mathcal{G}_{\varphi}$. For Inequality~\eqref{ineq2}, note that by $\mathcal{G}_{\varphi}$ we have $V_a(x)\geq V_b(A_1 x)$ and $V_b(x)\geq V_a (A_2 x)$, thus, $V_a(x)\geq V_a(A_2 A_1 x)$, which is equivalent to ${V}_{a'}(x)\geq {V}_{b'}(A_1 x)$. Similarly, for Inequality~\eqref{ineq3} note that by $\mathcal{G}_{\varphi}$ we have $V_a(x)\geq V_c(A_2 x)$ and $V_c(x)\geq V_a (A_2 x)$, thus, $V_a(x)\geq V_a(A_2 A_2 x)$, which is equivalent to ${V}_{a'}(x)\geq {V}_{b'}(A_2 x)$. At last, for Inequality~\eqref{ineq4}, note that ${V}_{b'}(x)=V_a (A_2 x)= {V}_{a'} (A_2 x)$. This concludes Step 1, next we show that $\mathcal{G}_{\varphi}^{\circ}$ does \textit{not} simulate $\mathcal{G}_{\psi}$.

\textit{Step 2}. First note that $\mathcal{G}_{\varphi}=:(S_{\varphi},E_{\varphi})$ does not simulate $\mathcal{G}_{\psi}=:(S_{\psi},E_{\psi})$. For the sake of contradiction, assume a simulation relation does hold, thus, there must be a map $R:S_{\psi}\to S_{\varphi}$ such that for all $(a',b',i)\in E_{\psi}$ we have $(R(a'),R(b'),i)\in E_{\varphi}$.
As $(a',a',1)\in E_{\psi}$ we must have $a'\mapsto R(a'):=a$. However, under this $R$, the node $b'$ cannot map to $a$, since $(a',b',2)\in E_{\psi}$, yet, $(a,a,2)\not\in E_{\varphi}$. The node $b'$ can also not map to $b$, since $(a,b,2)\not\in E_{\varphi}$. At last, $b'$ cannot map to $c$ since $(a',b',1)\in E_{\psi}$, yet, $(a,c,1)\notin E_{\varphi}$. In conclusion, this simulation relation cannot exist. 

Next we claim that $\mathcal{G}_{\varphi}^{\circ}=:(S_{\varphi}^{\circ},E_{\varphi}^{\circ})$ does not simulate $\mathcal{G}_{\psi}$ either. Again, for the sake of contradiction, suppose such a simulation relation does hold. Then, there must be a map $R:S_{\psi} \to S_{\varphi}^{\circ}$ such that for all $(a',b',i)\in E_{\psi}$ we have $(R(a'),R(b'),i)\in E^{\circ}_{\varphi}$. In particular, note that we have both $(a',b',1)$ and $(a',b',2)$ in $E_{\psi}$. However, $\mathcal{G}^{\circ}_{\varphi}\setminus \mathcal{G}_{\varphi}$ never contains a pair of nodes with multiple and differently labeled edges between them. The reason is as follows, given any $(a,b,i)\in E_{\varphi}$, then, this edge is lifted to $((a,j),(b,i),j)$ for any $j\in \Sigma$, \textit{i.e.}, the node $(a,j)$ will only have outgoing edges with the label $j$ \textit{cf}. Definition~\ref{def:composition} and Example~\ref{ex:12lifts}. As this argument only holds true for $\mathcal{G}^{\circ}_{\varphi}\setminus \mathcal{G}_{\varphi}$, the remaining possibility for $\mathcal{G}_{\varphi}^{\circ}$ to simulate $\mathcal{G}_{\psi}$ is that $\mathcal{G}_{\varphi}$ simulates $\mathcal{G}_{\psi}$, but we just showed that this is impossible.   
\exampleEnd
\end{upshape}
\end{example}

We remark that in the above we did not exploit the \textit{linear} structure of $\mathcal{F}$, this was for simplicity. Also, without Assumption~\ref{ass:pclf} it is only easier to find counterexamples. 

Although Conjecture~\ref{conj:forward:composition} is not true, making the composition lift structurally different from, \textit{e.g.}, the sum lift, we observe the following structural benefit of the composition lift. For the $T$-sum lift, we have $\mathcal{G}\leq_{(\mathcal{V})}\mathcal{G}^{\oplus T}$ and $\mathcal{G}^{\oplus T}\leq \mathcal{G}$, for any $T\in \mathbb{N}_{>0}$ and any $\mathcal{V}$ closed under addition. Hence, for any path-complete graph $\mathcal{G}$, any $\mathcal{V}$ closed under addition and any $T\in \mathbb{N}_{>0}$ we have that $\rho_{\mathcal{G},\mathcal{V}}(\mathcal{A})=  \rho_{\mathcal{G}^{\oplus T},\mathcal{V}}(\mathcal{A})$, \textit{i.e.}, a $T$-sum lift will \textit{not} refine the JSR approximation. Now, for the (forward) composition lift, interestingly, $\mathcal{G}^{\circ T}\leq \mathcal{G}$ \textit{cannot} be guaranteed for $T\in \mathbb{N}_{>0}$ (\textit{e.g.}, see Example~\ref{ex:num} below for the formal justification of this claim) while $\mathcal{G}\leq_{(\mathcal{V},\mathcal{F})}\mathcal{G}^{\circ T}$, for $\mathcal{V}$ closed under composition with $\mathcal{F}$, \textit{is} true for any $T\in \mathbb{N}_{>0}$. Thus, non-trivial $T$-forward composition lifts of $\mathcal{G}$ are possibly \textit{strictly} better than $\mathcal{G}$ itself. In combination with Lemma~\ref{lem:connected}, this provides for an attractive refinement procedure. We illustrate this with numerical experiments.

\begin{example}[Refining graphs through the composition lift]
\label{ex:num}
\begin{upshape}
Consider the graph $\mathcal{G}_{\alpha}$ from Figure~\ref{fig:connected}. One can show that $\mathcal{G}_{\alpha}^{\circ 1}$ does not simulate $\mathcal{G}_{\alpha}$. Recall that $\mathcal{G}_{\alpha}\leq_{(\mathcal{V},\mathcal{F})}\mathcal{G}_{\alpha}^{\circ 1}$, for any $\mathcal{V}$ closed under the forward composition with the dynamics in $\mathcal{F}$, subject to those dynamics being invertible. Suppose we work with linear maps. We sample $N$ pairs $\mathcal{A}:=\{A_1,A_2\}$ from $\mathsf{GL}(n,\mathbb{R})$ such that $\rho(A_1)<1$ and $\rho(A_2)<1$ (\textit{i.e.}, we sample $\mathrm{vec}(A_i)$ from $\mathcal{N}(0,I_{n^2})$ until we have a feasible pair). Using~\eqref{equ:gamma:opt} and SDPT3 \cite{ref:tutuncu2003solving}, we compute $\rho_{\mathcal{G}_{\alpha},\mathcal{V}}(\mathcal{A})$ and $\rho_{\mathcal{G}^{\circ 1}_{\alpha},\mathcal{V}}(\mathcal{A})$ for $\mathcal{V}_n:=\{x\mapsto x^{\mathsf{T}}Px\,|\, P\succeq I_n\}$.  Indeed, $\rho_{\mathcal{G}^{\circ 1}_{\alpha},\mathcal{V}}(\mathcal{A})\leq \rho_{\mathcal{G}_{\alpha},\mathcal{V}}(\mathcal{A})$, but in particular, for $N=1000$ and $n=3$, $\rho_{\mathcal{G}^{\circ 1}_{\alpha},\mathcal{V}}(\mathcal{A})< \rho_{\mathcal{G}_{\alpha},\mathcal{V}}(\mathcal{A})$ for $12\%$ of the experiments, with $\rho_{\mathcal{G}_{\alpha},\mathcal{V}}(\mathcal{A})-\rho_{\mathcal{G}^{\circ 1}_{\alpha},\mathcal{V}}(\mathcal{A})=0.44$ on average. Similarly, if we compare $\mathcal{G}_0$ against $\mathcal{G}_0^{\circ 1}$ and $\mathcal{G}_0^{\circ 2}$ (\textit{e.g.}, recall Figures 1-3), in the same experimental setting, we find that $\rho_{\mathcal{G}^{\circ 1}_{0},\mathcal{V}}(\mathcal{A})< \rho_{\mathcal{G}_{0},\mathcal{V}}(\mathcal{A})$ for $12\%$ of the cases, with $\rho_{\mathcal{G}_{0},\mathcal{V}}(\mathcal{A})-\rho_{\mathcal{G}^{\circ 1}_{0},\mathcal{V}}(\mathcal{A})=0.43$ on average and $\rho_{\mathcal{G}^{\circ 2}_{0},\mathcal{V}}(\mathcal{A})< \rho_{\mathcal{G}^{\circ 1}_{0},\mathcal{V}}(\mathcal{A})$ for $5\%$ of the cases, with $\rho_{\mathcal{G}^{\circ 1}_{0},\mathcal{V}}(\mathcal{A})-\rho_{\mathcal{G}^{\circ 2}_{0},\mathcal{V}}(\mathcal{A})=0.42$ on average. We emphasize that $\mathcal{G}_{\alpha}^{\oplus T}$ or $\mathcal{G}_0^{\oplus T}$ provide no improvement, nor recipe. 
\exampleEnd
\end{upshape}
\end{example} 

\subsection{The transitive closure of the composition lift}
\label{sec:transitive}
Both Lemma~\ref{lem:converse} and Example~\ref{ex:conjecture:counter} indicate that we should \textit{combine} \textit{several} $T$-composition lifts. Indeed, if we have two inequalities $V_a(x)\geq V_b(f_i(x))$ and $V_b(x)\geq V_c(f_j(x))$ and work with a template closed under composition with $f_i$ and $f_j$, then we get $V_a(x)\geq V_{(c,j)}(f_i(x))$ for free. By identifying any expression of the form $V_{(s,i_1,\dots,i_T)}(f_k(x))$ with $V_{(s,i_1,\dots,i_T,k)}(x)$ (\textit{i.e.}, allow for the moment for edges without labels), we can define $(\mathcal{G}^{\circ})^+$ as the \textit{transitive closure} of $\mathcal{G}^{\circ}$ with respect to the relation $\leq$ on $\mathbb{R}_{\geq 0}$ (not to be confused with a preorder for graphs), where we only keep labeled edges with words of length one, see Figure~\ref{fig:trans}. Importantly, equalities of the form $V_{(a,i)}(x)=V_a(f_i(x))$ will be understood to give rise to $V_{(a,i)}(x)\geq V_a(f_i(x))$. Doing so, we readily find that $(\mathcal{G}_{\varphi}^{\circ})^+$ \textit{does} simulate $\mathcal{G}_{\psi}$, see Figure~\ref{fig:conjecture:counter:trans}. A full formal study of this new lift is future work, but the following result is immediate and now applicable to Example~\ref{ex:conjecture:counter}.
\begin{theorem}[The transitive composition lift]
    Consider $\mathcal{G}, \widetilde{\mathcal{G}}\in \mathrm{pc}(\Sigma)$. Suppose that $\mathcal{F}$ comprises a set of continuous invertible maps, then
\begin{enumerate}[(i)]
    \item $(\mathcal{G}^{\circ})^+$ simulates $\widetilde{\mathcal{G}}$ $\implies$
    \item $\mathcal{G}\leq_{\mathcal{V},\mathcal{F}} \widetilde{\mathcal{G}}$ for any template $\mathcal{V}$ closed under forward composition with the class of dynamics $\mathcal{F}$.
\end{enumerate}
\end{theorem}

\begin{figure}
    \centering
   \includegraphics[scale=0.6]{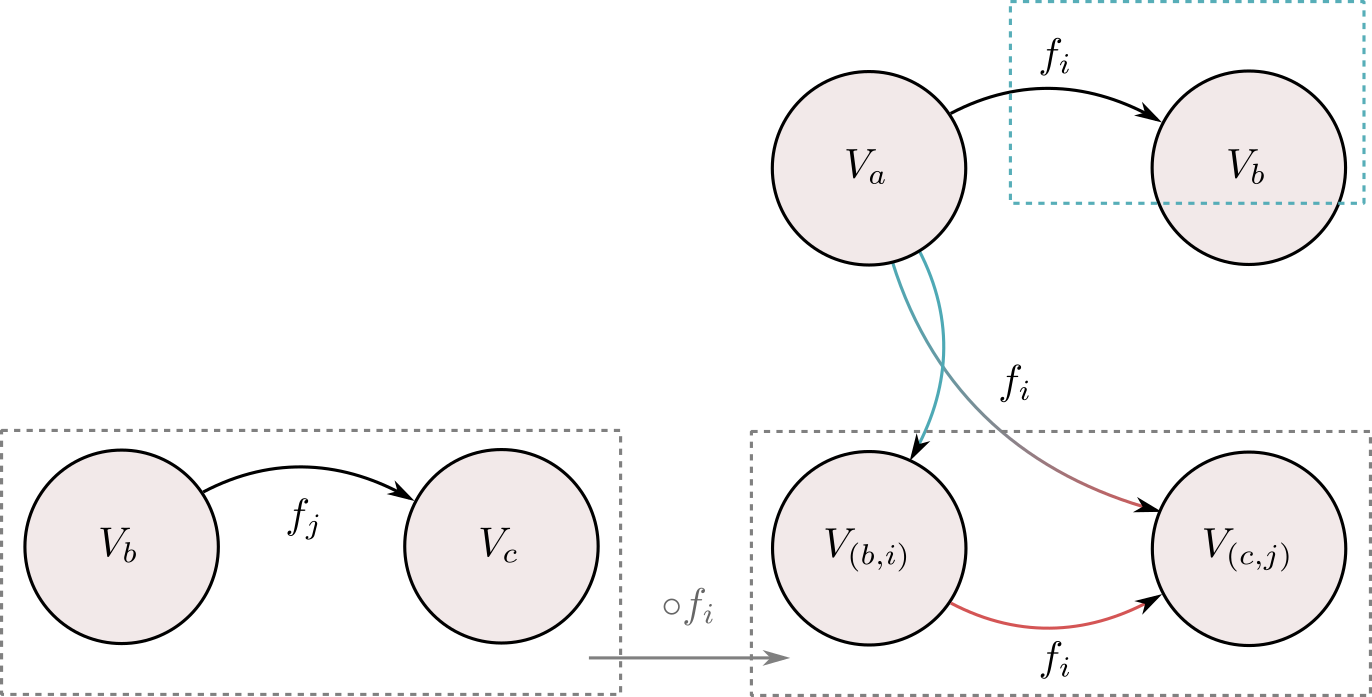}
    \caption{Example construction of the transitive closure of a composition lift: start from $V_a(x)\geq V_b(f_i(x))$ and $V_b(x)\geq V_c(f_j(x))$, lift the second inequality to $V_{(b,i)}(x)\geq V_{(c,j)}(f_i(x))$ and see that $V_a(x)\geq V_{(b,i)}(x)$. Combine the aforementioned to construct the edge $V_a(x)\geq V_{(c,j)}(f_i(x))$.}
    \label{fig:trans}
\end{figure}

\begin{figure}
    \centering
   \includegraphics[scale=0.6]{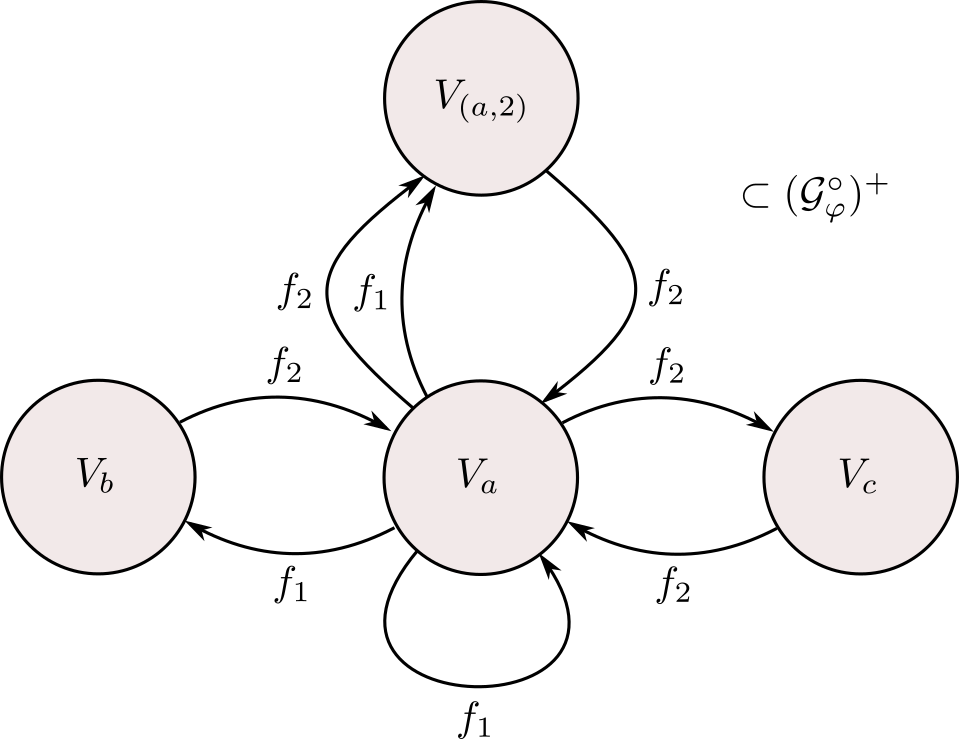}
    \caption{Reconsider the graphs $\mathcal{G}_{\varphi}$ and $\mathcal{G}_{\psi}$ from Example~\ref{ex:conjecture:counter}. Although $\mathcal{G}^{\circ}_{\varphi}$ does not simulate $\mathcal{G}_{\psi}$, $(\mathcal{G}^{\circ}_{\varphi})^+$ does.}
    \label{fig:conjecture:counter:trans}
\end{figure}

\section{Conclusion and future work}
We have motivated the need for new tools to compare Lyapunov inequalities (see Example~\ref{ex:counter:T-sum}). To that end, we have studied the composition lift, showed some desirable properties (see Section~\ref{sec:comp:lift:properties} and Example~\ref{ex:num}), but also that a key conjecture is false (see Example~\ref{ex:conjecture:counter}). Then, to address that $\mathcal{G}_{\varphi}$ and $\mathcal{G}_{\psi}$ \textit{can} be preordered, just not via a composition lift, we proposed to use the transitive closure of the composition lift (see Section~\ref{sec:transitive}), plus we recently started to study abstract lifts. We can show that given any preorder, a lift with properties akin to Theorem~\ref{theorem:T:sum:lift} always \textit{exists}~\cite{ref:JongeneelJungersBNLX25}. Other avenues of interest are statistical orderings~\cite{ref:jungers2024statistical} and certain \textit{equivalence classes} of PCLFs.

\printbibliography
\end{document}